\documentclass[12pt,reqno]{amsart}
\usepackage{amsmath, amsfonts, amssymb, amsthm, hyperref}
\usepackage{bm}
\allowdisplaybreaks[4]
\def\l{\left}
\def\r{\right}
\def\bg{\bigg}
\def\({\bg(}
\def\){\bg)}

\def\f{\frac}

\def\eq{\equiv}

\def\Z{\mathbb Z}
\def\C{\mathbb C}
\def\N{\mathbb N}

\def\R{\mathbb R}

\def\<{\langle}
\def\>{\rangle}
\def\1{{\bf 1}}

\theoremstyle{plain}
\newtheorem{theorem}{Theorem}[section]
\newtheorem{lemma}{Lemma}[section]

\newtheorem{conjecture}{Conjecture}[section]

\theoremstyle{definition}

\newtheorem*{Acks}{Acknowledgments}
\theoremstyle{remark}
\newtheorem{remark}{Remark}

\numberwithin{equation}{section}

\begin{document}
\hbox{}
\medskip

\title[Some congruences from the Karlsson-Minton summation formula]{Some congruences from the Karlsson-Minton\\ summation formula}

\author{Junhang Li}
\address[Junhang Li]{Department of Applied Mathematics, Nanjing Forestry
University, Nanjing 210037, People's Republic of China}
\email{745679122@qq.com}

\author{Yezhenyang Tang}
\address[Yezhenyang Tang]{Department of Applied Mathematics, Nanjing Forestry
University, Nanjing 210037, People's Republic of China}
\email{2742425229@qq.com}

\author{Chen Wang*}
\address[Chen Wang]{Department of Applied Mathematics, Nanjing Forestry
University, Nanjing 210037, People's Republic of China}
\email{cwang@smail.nju.edu.cn}

\thanks{*Chen Wang is the corresponding author}

\begin{abstract}
Let $p$ be an odd prime. In this paper, by using the well-known Karlsson-Minton summation formula, we mainly prove two supercongruences as variants of a supercongruence of Deines-Fuselier-Long-Swisher-Tu, which confirm some recent conjectures of V.J.W. Guo.
\end{abstract}
\maketitle

\section{Introduction}

For any $n\in\N=\{0,1,2,\ldots\}$, let $(x)_n=x(x+1)\cdots(x+n-1)$ denote the Pochhammer symbol. For $n,r\in\N$ and $a_0,\ldots,a_r,b_1,\ldots,b_r,z\in\C$ with $(b_1)_n,\ldots,(b_r)_n$ being nonzero, the truncated hypergeometric series ${}_{r+1}F_r$ are defined as
$$
{}_{r+1}F_r\bigg[\begin{matrix}a_0,&a_1,&\ldots,&a_r\\ &b_1,&\ldots,&b_r\end{matrix}\bigg|\ z\bigg]_n=\sum_{k=0}^{n}\f{(a_0)_k\cdots(a_r)_k}{(b_1)_k\cdots(b_r)_k}\cdot\f{z^k}{k!}.
$$
Clearly, they are partial sums of the classical hypergeometric series. Let $p$ be an odd prime. For any integer $n\geq1$, the $p$-adic Gamma function introduced by Morita (cf. \cite{Morita,R}) is defined as
$$
\Gamma_p(n)=(-1)^n\prod_{\substack{1\leq k<n\\ p\nmid k}}k.
$$
Moreover, set $\Gamma_p(0)=1$, and for any $p$-adic integer $x$ set
$$
\Gamma_p(x)=\lim_{n\to x}\Gamma_p(n),
$$
where $n$ runs through any sequence of positive integers $p$-adically approaching $x$.

Rodriguez-Villegas \cite{RVillegas03} investigated hypergeometric families of Calabi-Yau manifolds, and discovered (numerically) a number of possible supercongruences. Some of them have been proved in \cite{Mortenson03,Mortenson04} where Mortenson, with the help of the Gross-Koblitz formula,  determined ${}_{2}F_1\big[\genfrac{}{}{0pt}{}{\alpha}{}\genfrac{}{}{0pt}{}{1-\alpha}{1}\big|\,1\big]_{p-1}$ modulo $p^2$ for $\alpha\in\{1/2,1/3,1/4,1/6\}$. For instance, he showed that
\begin{equation}\label{RVcon}
{}_{2}F_1\bigg[\begin{matrix}\frac12,&\frac12\\ &1\end{matrix}\bigg|\,1\bigg]_{p-1}\equiv(-1)^{(p-1)/2}\pmod{p^2}
\end{equation}
for any prime $p\geq 5$. Using the Legendre relation of $p$-adic Gamma function (cf. \cite[p. 370]{R}), we may replace the right-hand side of \eqref{RVcon} with $-\Gamma_p(1/2)^2$. Later, Sun \cite{SunZH14} extended Mortenson's result to the general $p$-adic integer $\alpha$. Let $\Z_p$ denote the ring of all $p$-adic integers and $\Z_p^\times:=\{x\in\Z_p:\,p\nmid x\}$. Z.-H. Sun proved that for each odd prime $p$ and $\alpha\in\Z_p^\times$,
\begin{equation}\label{ZHSun}
{}_{2}F_1\bigg[\begin{matrix}\alpha,&1-\alpha\\ &1\end{matrix}\bigg|\,1\bigg]_{p-1}\equiv(-1)^{\langle-\alpha\rangle_p}\pmod{p^2},
\end{equation}
where $\langle x\rangle_p$ is the least nonnegative residue of $x$ modulo $p$, i.e., $0\leq \langle x\rangle_p\leq p-1$ and $x\equiv \langle x\rangle_p\pmod{p}$. On the other hand, Deines et al. \cite{DFLST16} obtained the following generalization of \eqref{RVcon}: for any integer $d>1$ and prime $p\eq1\pmod{d}$,
\begin{equation}\label{longcon}
{}_dF_{d-1}\bigg[\begin{matrix}1-\f1d,&1-\f1d,&\ldots,&1-\f1d\\ &1,&\ldots,&1\end{matrix}\bigg|\ 1\bigg]_{p-1}\eq-\Gamma_p\l(\f1d\r)^d\pmod{p^2}.
\end{equation}
In fact, Deines et al. also conjectured that for any integer $d\geq3$ and prime $p\eq1\pmod{d}$, the congruence \eqref{longcon} holds modulo $p^3$, and this conjecture was later confirmed by the third author and Pan \cite{WangPan}.

In the past decade, many mathematicians studied $q$-analogues of \eqref{RVcon} and its generalizations; among these $q$-congruences, the first one was obtained by Guo and Zeng \cite{GuoZeng}. Recently, via the so-called `creative microscoping' method introduced by Guo and Zudilin \cite{GuoZu}, Guo \cite{Guo2022} established a $q$-analogue of \eqref{longcon}. Meanwhile, Guo obtained a variant of \eqref{longcon} as follows: for any integer $d>1$ and prime $p\eq1\pmod{d}$,
$$
\sum_{k=0}^{p-1}\f{k(\f{d-1}{d})_k^d}{k!^d}\eq \f{(d-1)\Gamma_p(\f1d)^d}{2d}\pmod{p^2}.
$$

The main purpose of this paper is to prove the following variants of \eqref{longcon} which confirm two conjectures of Guo \cite[(5.4) and (5.5)]{Guo2022}.

\begin{theorem}\label{mainth1}
{\rm (i)} Let $d\geq4$ be an even integer. Then, for any prime $p\eq-1\pmod{d}$ with $p\geq 2d-1$,
\begin{equation}\label{mainth1eq1}
{}_dF_{d-1}\bigg[\begin{matrix}\f1d-1,&1+\f1d,&1+\f1d,&\ldots,&1+\f1d\\ &1,&1,&\ldots,&1\end{matrix}\bigg|\ 1\bigg]_{p-1}\eq\f{d-1}{d^2}\Gamma_p\l(-\f1d\r)^d\pmod{p^2}.
\end{equation}

{\rm (ii)} Let $d\geq3$ be an odd integer. Then, for any prime $p\eq-1\pmod{d}$,
\begin{equation}\label{mainth1eq2}
{}_dF_{d-1}\bigg[\begin{matrix}\f1d,&\f1d,&1+\f1d,&1+\f1d,&\ldots,&1+\f1d\\ &1,&1,&1,&\ldots,&1\end{matrix}\bigg|\ 1\bigg]_{p-1}\eq-\f{1}{d^2}\Gamma_p\l(-\f1d\r)^d\pmod{p^2}.
\end{equation}
\end{theorem}

The second goal is to show another conjectural congruence posed by Guo \cite[Conjecture 1.3]{Guo2022}.

\begin{theorem}\label{mainth2}
Let $p\eq1\pmod4$ be a prime and $r\geq1$. Then
\begin{equation}\label{mainth2eq}
\sum_{k=0}^{p^r-1}\l(k-\f{p^{2r}-1}{4}\r)\f{(\f12)_k^2}{k!^2}\eq0\pmod{p^{2r+1}}.
\end{equation}
\end{theorem}

The rest of this paper is organized as follows. In the next section, we list some necessary lemmas which play key roles in the proof of Theorem \ref{mainth1}. Section 3 is devoted to the proof of Theorem \ref{mainth1}. In Section 4, we prove Theorem \ref{mainth2}. In Section 5, we shall pose a conjecture for further research.

\medskip

\section{Some necessary lemmas}
\setcounter{lemma}{0}
\setcounter{theorem}{0}
\setcounter{equation}{0}
\setcounter{conjecture}{0}
\setcounter{remark}{0}

The first key ingredient of our proofs is the following Karlsson-Minton summation formula (cf. \cite[p. 19]{GR}).

\begin{lemma}\label{kmiden}
Let $m_1,m_2,\ldots,m_n$ be nonnegative integers. Then
\begin{align}\label{kmidentype1}
&{}_{n+1}F_{n}\bigg[\begin{matrix} -(m_1+\cdots+m_n),&b_1+m_1,&\ldots,&b_n+m_n,\\ &b_1,&\ldots,&b_n\end{matrix}\bigg | \, 1\bigg]\notag
\\&\qquad=(-1)^{m_1+\cdots+m_n}\cdot\f{(m_1+\cdots+m_n)!}{(b_1)_{m_1}\cdots(b_n)_{m_n}}.
\end{align}

\end{lemma}

Our proofs also rely on some properties of the $p$-adic Gamma functions.

\begin{lemma}\cite[p. 369]{R}\label{padicgammalem1}
Let $p$ be an odd prime and $x\in\Z_p$. Then
\begin{gather}
\label{Gammapxx1}\frac{\Gamma_p(x+1)}{\Gamma_p(x)}=\begin{cases}-x,&\text{if }p\nmid x,\\
-1,&\text{if }p\mid x,\end{cases}\\
\vspace{0.2mm}
\label{Gammapxx2}\Gamma_p(x)\Gamma_p(1-x)=(-1)^{\<-x\>_p-1}.
\end{gather}
\end{lemma}

\begin{remark}
{\rm(a)} By \eqref{Gammapxx1}, it is easy to see that for any positive integer $n\leq p$,
\begin{equation}\label{Gammapxx1cor}
\Gamma_p(n)=(-1)^n\Gamma(n).
\end{equation}

{\rm(b)} The identity \eqref{Gammapxx2} is a $p$-adic analogue of the following Legendre relation of the classical Gamma function:
$$
\Gamma(x)\Gamma(1-x)=\f{\pi}{\sin\pi x}.
$$
\end{remark}

The next lemma concerns a $p$-adic approximation to $\Gamma_p$-quotients.

\begin{lemma}\cite[Theorem 14]{LoRa16}\label{padicgammalem2}
For any prime $p\geq5$ and $x\in\Z_p$, there exists $G_1(x)\in\Z_p$ such that for any $t\in\Z_p$,
\begin{equation}\label{padicgammalem2eq}
\Gamma_p(x+tp)\eq\Gamma_p(x)(1+G_1(x)tp)\pmod{p^2}.
\end{equation}
\end{lemma}

For the properties of $G_1(x)$, the reader may consult \cite{PTW}. The following lemma lists two identities involving the derivatives of $(1+\alpha+x)_k$, which can be verified directly.

\begin{lemma}\label{lem4}
For any integer $k\geq 0$ and $\alpha,\beta\in\R$,
\begin{align*}
\frac{d}{d x}(1+\alpha+x)_k=&(1+\alpha+x)_k\sum_{j=1}^{k}\frac{1}{j+\alpha+x},\\
\frac{d}{d x}\bigg(\frac1{(1+\beta+x)_k}\bigg)=&-\frac1{(1+\beta+x)_k}\sum_{j=1}^{k}\frac{1}{j+\beta+x}.
\end{align*}
\end{lemma}

\medskip

\section{Proof of Theorem \ref{mainth1}}

Throughout this section, we set $m=(p+1)/d$. We first prove \eqref{mainth1eq1}. To show \eqref{mainth1eq1}, we need the following preliminary result.

\begin{lemma}\label{mainth1eq1lem}
Under the assumptions of Theorem \ref{mainth1} (i), modulo $p$, we have
$$
\sum_{k=0}^{p-1}\f{(m-1)_k(m+1)_k^{d-1}}{(1)_k^d}\l(\sum_{j=0}^{k-1}\f{1}{m-1+j}-\sum_{j=0}^{k-1}\f{1}{m+1+j}\r)\eq \f{(p-1)!}{(1)_{m-2}(1)_m^{d-1}}\l(\f{1}{m}+\f{1}{m-1}\r).
$$
\end{lemma}
\begin{proof}
For $x,y\in(-1,+\infty)$, set
\begin{equation}\label{Psi1}
\Psi(x,y)={}_{d+1}F_d\bigg[\begin{matrix}1-p,&m-1+x,&m+1+y,&m+1,&\ldots,&m+1\\ &1+x,&1+y,&1,&\ldots,&1\end{matrix}\bigg|\ 1\bigg]_{p-1}.
\end{equation}
Clearly, $\Psi(x,y)$ is smooth on $(-1,+\infty)\times(-1,+\infty)$. Since $(1-p)_k=0$ for all $k\geq p$, we have
$$
\Psi(x,y)={}_{d+1}F_d\bigg[\begin{matrix}1-p,&m-1+x,&m+1+y,&m+1,&\ldots,&m+1\\ &1+x,&1+y,&1,&\ldots,&1\end{matrix}\bigg|\ 1\bigg].
$$
As
$$
m-2+(d-1)m=m-2+p+1-m=p-1,
$$
by Lemma \ref{kmiden},
\begin{equation}\label{Psi2}
\Psi(x,y)=\f{(p-1)!}{(1+x)_{m-2}(1+y)_m(1)_m^{d-2}}.
\end{equation}

Now we calculate $\Psi_x(0,0)-\Psi_y(0,0)$ in two different ways, where $\Psi_x(0,0)$ and $\Psi_y(0,0)$ stand for the partial derivatives of $\Psi$ at $(0,0)$ with respect to $x$ and $y$. By \eqref{Psi1} and Lemma \ref{lem4}, we obtain
\begin{align*}
\Psi_x(0,0)=&\sum_{k=0}^{p-1}\f{(1-p)_k(m-1)_k(m+1)_k^{d-1}}{(1)_k^{d+1}}\l(\sum_{j=0}^{k-1}\f1{m-1+j}-H_k\r),\\
\Psi_y(0,0)=&\sum_{k=0}^{p-1}\f{(1-p)_k(m-1)_k(m+1)_k^{d-1}}{(1)_k^{d+1}}\l(\sum_{j=0}^{k-1}\f1{m+1+j}-H_k\r),
\end{align*}
where $H_k=\sum_{j=1}^k1/j$ denotes the harmonic number. Therefore,
\begin{equation}\label{Psilkey1}
\Psi_x(0,0)-\Psi_y(0,0)=\sum_{k=0}^{p-1}\f{(1-p)_k(m-1)_k(m+1)_k^{d-1}}{(1)_k^{d+1}}\l(\sum_{j=0}^{k-1}\f1{m-1+j}-\sum_{j=0}^{k-1}\f1{m+1+j}\r).
\end{equation}
On the other hand, by \eqref{Psi2} and Lemma \ref{lem4},
\begin{equation}\label{Psilkey2}
\Psi_x(0,0)-\Psi_y(0,0)=\f{(p-1)!}{(1)_{m-2}(1)_m^{d-1}}\l(\f1m+\f1{m-1}\r).
\end{equation}
Note that for $k$ among $0,1,\ldots,p-1$, $(1-p)_k\eq(1)_k\pmod{p}$ and
$$
\f{(m-1)_k(m+1)_k^{d-1}}{(1)_k^{d+1}}\l(\sum_{j=0}^{k-1}\f1{m-1+j}-\sum_{j=0}^{k-1}\f1{m+1+j}\r)\in\Z_p.
$$
This, together with \eqref{Psilkey1} and \eqref{Psilkey2}, gives the desired result.
\end{proof}

\medskip

\noindent{\it Proof of \eqref{mainth1eq1}}. For any $x,y\in\Z_p$, let
$$
\Phi(x,y)={}_dF_{d-1}\bigg[\begin{matrix}m-1+x,&m+1+y,&\ldots,&m+1+y\\&1,&\ldots,&1\end{matrix}\bigg|\ 1\bigg]_{p-1}.
$$
Obviously, for any $s,t\in\Z_p$,
\begin{equation}\label{Phikey1}
\Phi(sp,tp)\eq\Phi(0,0)+sp\Phi_x(0,0)+tp\Phi_y(0,0)\pmod{p^2}.
\end{equation}
In particular,
\begin{equation}\label{Phikey2}
\Phi(-p,0)\eq\Phi(0,0)-p\Phi_x(0,0)\pmod{p^2}.
\end{equation}
Substituting \eqref{Phikey2} into \eqref{Phikey1}, we get
$$
\Phi(sp,tp)\eq\Phi(-p,0)+(s+1)p\Phi_x(0,0)+tp\Phi_y(0,0)\pmod{p^2}.
$$
Taking $s=t=-1/d$, in view of Lemmas \ref{lem4} and \ref{mainth1eq1lem},
\begin{align}\label{Phikey3}
&\Phi\l(-\f{p}d,-\f{p}d\r)\\
&\qquad={}_dF_{d-1}\bigg[\begin{matrix}\f1d-1,&1+\f1d,&1+\f1d,&\ldots,&1+\f1d\\ &1,&1,&\ldots,&1\end{matrix}\bigg|\ 1\bigg]_{p-1}\notag\\
&\qquad\eq\Phi(-p,0)+\l(1-\f1d\r)p\Phi_x(0,0)-\f{p}d\Phi_y(0,0)\notag\\
&\qquad\eq\Phi(-p,0)+\l(1-\f1d\r)p\sum_{k=0}^{p-1}\f{(m-1)_k(m+1)_k^{d-1}}{(1)_k^d}\l(\sum_{j=0}^{k-1}\f{1}{m-1+j}-\sum_{j=0}^{k-1}\f{1}{m+1+j}\r)\notag\\
&\qquad\eq\Phi(-p,0)+\l(1-\f1d\r)p\cdot\f{(p-1)!}{(1)_{m-2}(1)_m^{d-1}}\l(\f{1}{m}+\f{1}{m-1}\r)\pmod{p^2}.
\end{align}

Now we evaluate $\Phi(-p,0)$ modulo $p^2$. Since $m(d-1)=p+1-m$, by Lemma \ref{kmiden},
$$
\Phi(-p,0)=(-1)^m\f{(p+1-m)!}{(1)_m^{d-1}}.
$$
With the help of \eqref{Gammapxx1}, we obtain
\begin{align}\label{Phikey4}
\Phi(-p,0)=&(-1)^m\f{\Gamma(p-m+2)}{\Gamma(m+1)^{d-1}}=(-1)^{p+2-(m+1)(d-1)}\f{\Gamma_p(p-m+2)}{\Gamma_p(m+1)^{d-1}}\notag\\
=&(-1)^{1-(m+1)(d-1)}(p-m)(p-m+1)\f{\Gamma_p(p-m)}{\Gamma_p(m+1)^{d-1}}\notag\\
=&(-1)^{m}(p-m)(p-m+1)\f{\Gamma_p(p-m)}{\Gamma_p(m+1)^{d-1}},
\end{align}
where in the last step we have used the fact that $d$ is even. In light of \eqref{Gammapxx2} and Lemma \ref{padicgammalem2},
\begin{align}\label{Phikey5}
\f{\Gamma_p(p-m)}{\Gamma_p(m+1)^{d-1}}=&(-1)^{(m-1)(d-1)}\Gamma_p(p-m)\Gamma_p(-m)^{d-1}\notag\\
\eq&(-1)^{m-1}\Gamma_p\l(-\f1d\r)^d\l(1+\l(1-\f1d\r)pG_1\l(-\f1d\r)\r)\l(1-\f{p}dG_1\l(-\f1d\r)\r)^{d-1}\notag\\
\eq&(-1)^{m-1}\Gamma_p\l(-\f1d\r)^d\pmod{p^2}.
\end{align}
Moreover,
\begin{align*}
(p-m)(p-m+1)=&\l(-\f1d+\l(1-\f1d\r)p\r)\l(1-\f1d+\l(1-\f1d\r)p\r)\\
\eq&\f{1-d}{d^2}+\f{(d-1)(d-2)p}{d^2}\pmod{p^2}.
\end{align*}
Combining this with \eqref{Phikey4} and \eqref{Phikey5}, we obtain
\begin{equation}\label{maintheq1key1}
\Phi(-p,0)\eq\l(\f{d-1}{d^2}-\f{(d-1)(d-2)p}{d^2}\r)\Gamma_p\l(-\f1d\r)^d\pmod{p^2}.
\end{equation}
Similarly, it is routine to verify that
\begin{equation}\label{maintheq1key2}
\l(1-\f1d\r)p\cdot\f{(p-1)!}{(1)_{m-2}(1)_m^{d-1}}\l(\f{1}{m}+\f{1}{m-1}\r)\eq\f{(d-1)(d-2)p}{d^2}\Gamma_p\l(-\f1d\r)^d\pmod{p^2}.
\end{equation}

Substituting \eqref{maintheq1key1} and \eqref{maintheq1key2} into \eqref{Phikey3}, we immediately get \eqref{mainth1eq1}.\qed

\medskip

In order to show \eqref{mainth1eq2}, we also need an auxiliary lemma.

\begin{lemma}\label{mainth1eq2lem}
Under the assumptions of Theorem \ref{mainth1} (ii), we have
$$
\sum_{k=0}^{p-1}\f{(m)_k^2(m+1)_k^{d-2}}{(1)_k^d}\l(\sum_{j=0}^{k-1}\f{1}{m+j}-\sum_{j=0}^{k-1}\f{1}{m+1+j}\r)\eq \f{(p-1)!}{(1)_{m-1}(1)_m^{d-1}}\pmod{p}.
$$
\end{lemma}
\begin{proof}
For $x,y\in(-1,+\infty)$, set
\begin{align*}
\Upsilon(x,y)={}_{d+1}F_d\bigg[\begin{matrix}1-p,&m+x,&m,&m+1+y,&m+1,&\ldots,&m+1\\ &1+x,&1,&1+y,&1,&\ldots,&1\end{matrix}\bigg|\ 1\bigg]_{p-1}.
\end{align*}
Similarly as in the proof of Lemma \ref{mainth1eq1lem}, we are led to the desired result by considering $\Upsilon_x(0,0)-\Upsilon_y(0,0)$.
\end{proof}

\medskip

\noindent{\it Proof of \eqref{mainth1eq2}}. For any $x,y\in\Z_p$, set
$$
\Omega(x,y)={}_dF_{d-1}\bigg[\begin{matrix}m+x,&m+y,&m+1+y,&\ldots,&m+1+y\\&1,&1,&\ldots,&1\end{matrix}\bigg|\ 1\bigg]_{p-1}.
$$
Similarly as before, by Lemmas \ref{kmiden}, \ref{lem4} and \ref{mainth1eq2lem},
\begin{align}\label{Omegakey3}
&\Omega\l(-\f{p}d,-\f{p}d\r)\\
&\qquad={}_dF_{d-1}\bigg[\begin{matrix}\f1d,&\f1d,&1+\f1d,&1+\f1d,&\ldots,&1+\f1d\\ &1,&1,&1,&\ldots,&1\end{matrix}\bigg|\ 1\bigg]_{p-1}\notag\\
&\qquad\eq\Omega(-p,0)+\l(1-\f1d\r)p\Omega_x(0,0)-\f{p}d\Omega_y(0,0)\notag\\
&\qquad\eq\Omega(-p,0)+\l(1-\f2d\r)p\sum_{k=0}^{p-1}\f{(m)_k^2(m+1)_k^{d-2}}{(1)_k^d}\l(\sum_{j=0}^{k-1}\f{1}{m+j}-\sum_{j=0}^{k-1}\f{1}{m+1+j}\r)\notag\\
&\qquad\eq\Omega(-p,0)+\l(1-\f2d\r)p\cdot\f{(p-1)!}{(1)_{m-1}(1)_m^{d-1}}\pmod{p^2}.
\end{align}
Then, in view of Lemmas \ref{kmiden}--\ref{padicgammalem2}, we have
$$
\Omega(-p,0)\eq\l(-\f1{d^2}+\f{(d-2)p}{d^2}\r)\Gamma_p\l(-\f1d\r)^d\pmod{p^2}
$$
and
$$
\l(1-\f2d\r)p\cdot\f{(p-1)!}{(1)_{m-1}(1)_m^{d-1}}\eq\f{(2-d)p}{d^2}\cdot\Gamma_p\l(-\f1d\r)^d\pmod{p^2}.
$$

The proof of \eqref{mainth1eq2} follows by combining the above.\qed

\medskip

\section{Proof of Theorem \ref{mainth2}}

We need the following identity which can be verified by induction on $n$.
\begin{lemma}\label{th2id}
For any positive integer $n$, we have
$$
\sum_{k=0}^{n-1}(4k+1)\f{(\f12)_k^2}{k!^2}=\f{n^2}{4^{2n-1}}\binom{2n}{n}^2.
$$
\end{lemma}

Assuming $p\eq1\pmod{4}$ and putting $x=1/2$ in \cite[Corollary 1.4]{Liu2017}, we have the following result.
\begin{lemma}\label{Liu}
Let $p\eq1\pmod4$ be a prime and $r$ a positive integer. Then
$$
\sum_{k=0}^{p^r-1}\f{(\f12)_k^2}{k!^2}\eq1\pmod{p^2}.
$$
\end{lemma}

\medskip

\noindent{\it Proof of Theorem \ref{mainth2}}. Equivalently, we only need to show
\begin{equation}\label{mainth2key}
\sum_{k=0}^{p^r-1}(4k+1)\f{(\f12)_k^2}{k!^2}\eq p^{2r}\sum_{k=0}^{p^r-1}\f{(\f12)_k^2}{k!^2}\pmod{p^{2r+1}}.
\end{equation}
Putting $n=p^{2r}$ in Lemma \ref{th2id}, we obtain
\begin{equation}\label{lhs}
\sum_{k=0}^{p^r-1}(4k+1)\f{(\f12)_k^2}{k!^2}=\f{p^{2r}}{4^{2p^r-1}}\binom{2p^r}{p^r}^2.
\end{equation}
By Fermat's little theorem,
\begin{equation}\label{flt}
4^{2p^r-1}=4\times16^{p^r-1}=4\times(16^{p-1})^{\f{p^r-1}{p-1}}\eq4\pmod{p}.
\end{equation}
Using Lucas's theorem, we have
\begin{equation}\label{lt}
\binom{2p^r}{p^r}\eq\binom{2p^{r-1}}{p^{r-1}}\eq\cdots\eq\binom{2}{1}=2\pmod{p}.
\end{equation}
Substituting \eqref{flt} and \eqref{lt} into \eqref{lhs}, the left-hand side of \eqref{mainth2key} becomes $p^{2r}$ modulo $p^{2r+1}$. In view of Lemma \ref{Liu}, we arrive at Theorem \ref{mainth2} at once.\qed

\medskip
\section{Concluding remarks}

It is routine to check that
\begin{align}\label{combin}
&{}_dF_{d-1}\bigg[\begin{matrix}\f1d-1,&1+\f1d,&\ldots,&1+\f1d\\ &1,&\ldots,&1\end{matrix}\bigg|\ 1\bigg]_{p-1}+(d-1){}_dF_{d-1}\bigg[\begin{matrix}\f1d,&\f1d,&1+\f1d,&\ldots,&1+\f1d\\ &1,&1,&\ldots,&1\end{matrix}\bigg|\ 1\bigg]_{p-1}\notag\\
&\qquad=d\cdot{}_dF_{d-1}\bigg[\begin{matrix}\f1d-1,&\f1d,&1+\f1d,&1+\f1d,&\ldots,&1+\f1d\\ &1,&1,&1,&\ldots,&1\end{matrix}\bigg|\ 1\bigg]_{p-1}.
\end{align}
Note that Guo \cite[Corollaries 4.2 and 4.4]{Guo2022} proved that \eqref{mainth1eq1} and \eqref{mainth1eq2} also hold for odd integers $d\geq3$ and even integers $d\geq4$, respectively. This, together with \eqref{mainth1eq1}, \eqref{mainth1eq2} and \eqref{combin}, gives that
\begin{equation}\label{combincon}
{}_dF_{d-1}\bigg[\begin{matrix}\f1d-1,&\f1d,&1+\f1d,&1+\f1d,&\ldots,&1+\f1d\\ &1,&1,&1,&\ldots,&1\end{matrix}\bigg|\ 1\bigg]_{p-1}\eq0\pmod{p^2}
\end{equation}
for any integer $d\geq3$ and prime $p\eq-1\pmod{d}$ with $p\neq d-1$. In fact, \eqref{combincon} can also be proved independently by using the method we used to prove Theorem \ref{mainth1} and the following Karlsson-Minton summation formula:
$$
{}_{n+1}F_n\bigg[\begin{matrix}a,&b_1+m_1,&\ldots,&b_n+m_n\\&b_1,&\ldots,&b_n\end{matrix}\bigg|\ 1\bigg]=0
$$
provided that $m_1,\ldots,m_n$ are nonnegative integers and $\Re(-a)>m_1+\cdots+m_n$.

Motivated by \eqref{combincon} and based on some numerical calculations, we made the following conjecture for further study.

\begin{conjecture}\label{conj1}
Let $d\geq2$ be an integer. Let $n$ be a positive integer with $n\eq-1\pmod{d}$ and $n>d-1$. Then
\begin{equation}\label{conj1eq}
\f{(n-1)!^dd^{dn-d}}{n^2}\cdot {}_dF_{d-1}\bigg[\begin{matrix}\f1d-1,&\f1d,&1+\f1d,&1+\f1d,&\ldots,&1+\f1d\\ &1,&1,&1,&\ldots,&1\end{matrix}\bigg|\ 1\bigg]_{n-1}\in\Z.
\end{equation}
\end{conjecture}

We think it is possible that the `creative microscoping' could be used to prove Conjecture \ref{conj1}. We hope that an interested reader will make some progress on it.

\medskip

\begin{Acks}
We are very grateful to the anonymous referee for his/her valuable suggestions. This work is supported by the National Natural Science Foundation of China (grant no. 12201301) and the Jiangsu Province Student Innovation Training Program (grant no. 202210298128Y).
\end{Acks}

\noindent{\bf Declarations}. No potential conflict of interest was reported by the author. No availability of data and material. No code availability.

\end{document}